\documentclass[letterpaper, 10 pt, conference]{ieeeconf}  
\IEEEoverridecommandlockouts
\usepackage{graphicx,epstopdf}
\usepackage{epsfig} 
\usepackage{subfigure} 
\usepackage{amsmath,amssymb,bm,amsthm}
\usepackage{cite}
\usepackage{mathtools}
\usepackage{algpseudocode,algorithm,algorithmicx}

\usepackage[dvipsnames]{xcolor}
\usepackage[deletedmarkup=sout,authormarkup=superscript]{changes}
\usepackage{placeins}
\definechangesauthor[name={Jordan}, color=Orange]{JL}
\definechangesauthor[name={Frank}, color=Purple]{FP}
\definechangesauthor[name={Ilya}, color=Green]{IK}

\newcommand{\ints}{\mathbb{N}}
\newcommand\intrng[2]{\ints_{[#1,#2]}}

\newcommand{\mc}[1] {\mathcal{#1}}
\newcommand{\reals}[0] {\mathbb{R}}
\newcommand{\naturals}[0] {\mathbb{N}}
\newcommand{\posints}[0] {\mathbb{Z}_{+}}

\newcommand{\ones}{\mathbf{1}}

\newtheorem{ass}{Assumption}
\newtheorem{cor}{Corollary}
\newtheorem{lem}{Lemma}
\newtheorem{thm}{Theorem}
\newtheorem{rem}{Remark}

\newtheorem{prp}{Proposition}

\title{\LARGE \bf
A Computationally Governed Log-domain Interior-point Method for Model Predictive Control
}

\author{Jordan Leung$^{1}$, Frank Permenter$^{2}$, and Ilya V. Kolmanovsky$^{1}$
\thanks{Toyota Research Institute (TRI) provided funds to assist the authors with their research but this article solely reflects the opinions and conclusions of its authors and not TRI or any other Toyota entity. The third author acknowledges support by the National Science Foundation
award number CMMI-1904394.}
\thanks{$^{1}$ J. Leung and I. Kolmanovsky are with the University of Michigan, Ann Arbor, MI 48109, USA. Email: \texttt{\{jmleung, ilya\}@umich.edu}. $^{2}$ F. Permenter is with the Toyata Research Institute. Email:\texttt{frank.permenter@tri.global}.}}

\begin{document}

\maketitle
\thispagestyle{empty}
\pagestyle{empty}


\begin{abstract}
This paper introduces a computationally efficient approach for solving Model Predictive Control (MPC) reference tracking problems with state and control constraints. The approach consists of three key components: First, a log-domain interior-point quadratic programming method that forms the basis of the overall approach; second, a method of warm-starting this optimizer by using the MPC solution from the previous timestep; and third, a computational governor that bounds the suboptimality of the warm-start by altering the reference command provided to the MPC problem. As a result, the closed-loop system is altered in a manner so that MPC solutions can be computed using fewer optimizer iterations per timestep. In a numerical experiment, the computational governor reduces the worst-case computation time of a standard MPC implementation by 90\%, while maintaining good closed-loop performance.
\end{abstract}

\section{Introduction}

Model Predictive Control (MPC) is a feedback strategy defined by the solution of a receding horizon Optimal Control Problem (OCP). MPC is widely used in both industrial and academic settings since it provides  high-performance control  while directly accounting for constraints. Additionally,  a wide body of literature on the stability and robustness  of MPC is available   \cite{Rawlings2009_MPCBook,borelli_BBMBook,kouvaritakis2016model,goodwin_ConstrainedControl}. Unfortunately, implementing MPC can be challenging since a constrained OCP must  be solved at every timestep. The development of efficient methods for solving these OCPs has helped address this challenge \cite{Domahidi2013_SOCPSolver,Patrinos2014_APGM,Liao2020_FBStab,frison2020_HPIPM}, but this still remains an open problem in applications with fast sampling rates and/or limited computing power.  

A common approach to reduce the computational burden of MPC is to approximate the OCP solutions by performing a limited number of optimizer iterations per timestep - a procedure referred to as suboptimal MPC. While results pertaining to the stability and robustness of such methods do exist \cite{diehl2005_nominalstabofRTI,liao2019time,zanelli2020lyapunov}, computable certification bounds are limited to simple cases (e.g. input constrained linear systems  \cite{Liao-2020_LQMPCISS,leung2021_TDMPC-ROA_ACC,richter2011computational}) or require significant modification of the OCP (e.g. constraint tightening \cite{rubagotti2014_stabilizingLMPC}) to guarantee stability. 

Anytime MPC methods are a class of suboptimal strategies that ensure stabilizing solutions can be computed under arbitrary time constraints. The approach in \cite{Zeilinger2014_robustMPC} ensures that a specified warm-starting scheme always decreases a Lyapunov function through the addition of the so-called Lyapunov constraint to the OCP. Another approach uses relaxed barrier functions to ensure anytime stability at the cost of bounded constraint violation \cite{Feller2017_stabilizingIterations}.  Alternatively, the approach in \cite{bemporad2015_anytimeMPC} guarantees convergence to a terminal set by solving convex feasibility problems.

This paper introduces an alternative approach for achieving computationally efficient MPC. The  main novelty introduced in this paper is a computational governing scheme that alters the MPC reference command so that the  suboptimality of the interior-point optimizer initialization is bounded. In other words, the proposed strategy modifies the MPC problem to ensure that nearly optimal solutions can be computed using fewer interior-point optimizer iterations per timestep.

The paper is organized as follows. Section~\ref{sec:problemSetting} introduces the tracking MPC formulation used to generate control inputs. Section~\ref{sec:LDIPM} provides a brief overview of the log-domain interior-point method (LDIPM) used for optimization, while Section~\ref{sec:warmStart} describes how the LDIPM can be used to solve MPC problems. Section~\ref{sec:compGov} introduces the computational governing strategy and Section~\ref{sec:examples} demonstrates the efficacy of the computational governor through numerical experiments. 

\textbf{Notation:}    Let  $\mathbb{R}^{n}_{>0}$ denote the set of real  $n \times 1$ vectors,  with strictly positive elements (define  $\mathbb{R}^{n}_{\geq0}$ accordingly). Let $\posints = \mathbb{Z}_{\geq 0}$ represent the set of non-negative integers.  Given $a,b >0$, let $\naturals_{[a,b]} = \naturals \cap [a,b]$. Given $x \in \reals^n$ and $W \in \reals^{n\times n}$ with $W \succ 0$, let $\|x\|_W = \sqrt{x^T W x}$ denote the $W$-norm. Let $\| \cdot \|$  represent the 2-norm  when no subscript is specified. Let  $\lambda_-(A)$ and $\lambda_+(A)$ denote the minimum and maximum eigenvalues of $A\in\reals^{n\times n}$.  Given $x \in \reals^n$ and $y \in \reals^m$, let $(x~\circ~y) \in \reals^n$ denote the elementwise multiplication, and  $e^{x}$ and  $x^{-1}$ denote the elementwise exponentiation and inverse. Let $(x,y) = [x^T \ y^T]^T$. Given $x \in \reals^n$, let $\mathrm{diag}(x) = \mathrm{diag}(x_1,...,x_n)$ denote the $n\times n$ diagonal matrix containing $x_i$ in the $i^{\mathrm{th}}$ diagonal element $\forall i \in \naturals_{[1,n]}$.

\section{Problem Setting}
\label{sec:problemSetting}

Consider the following Linear Time Invariant (LTI) system 
\vspace{-\baselineskip}
\begin{subequations} \label{eq:LTI_System} 
\begin{align}
    x_{k+1} &= A x_k + B u_k, \\
    y_k &= C x_k + D u_k, \\
    z_k &= E x_k + F u_k,
\end{align}
\end{subequations}
where $k \in \posints$ is the discrete-time index, $x_k \in \reals^n$ is the state, $u_k \in \reals^{n_u}$ is the control, $y_k \in \reals^{n_y}$ is the constrained output, and $z_k \in \reals^{n_z}$ is the tracking output. The control objective is to drive the tracking output $z_k$ to a desired reference $r \in \reals^{n_z}$ subject to pointwise-in-time constraints 
\begin{equation} \label{eq:constraint}
    y_k \in \mathcal{Y}, \ \forall k \in \posints ,
\end{equation}
where $\mathcal{Y} \subseteq \reals^{n_y}$ is a specified constraint set. 
\begin{ass} \label{ass:AB-Stab}
   The pair $(A,B)$ is stabilizable and the constraint set $\mathcal{Y} = \{ y ~|~ Yy \leq h \}$ is a compact polyhedron that contains the origin in its interior.
\end{ass}

Equilibria of \eqref{eq:LTI_System} satisfy $Z [x^T \ u^T \ z^T]^T = 0$, where
\begin{equation}
    Z = \begin{bmatrix}
    A-I &  B &  0 \\
    E & F & -I
    \end{bmatrix},
\end{equation}
and these equilibria can be parameterized by a reference command $v \in \reals^{n_v}$ according to $(\bar{x}_v, \bar{u}_v, \bar{z}_v) = Gv$, where  $G^T = [G_x^T \ G_u^T \ G_z^T]^T$ is a basis for the nullspace of $Z$ and Assumption~\ref{ass:AB-Stab} ensures that $\mathrm{Null}(Z) \neq \{ 0 \}$ \cite{Skibik2021_feasibilityGov,Limon2008_trackingMPC}. The following assumption excludes ill-posed reference tracking problems, e.g., $G_z$ = 0, and ensures that the reference uniquely determines the target equilibrium.
\begin{ass} \label{ass:GZ-Nonsingular}
    The matrix $G_z$ is full rank and $n_z = n_v$. 
\end{ass}

\begin{rem}
    The results herein could be extended to the case where $G_z$ is full rank and $n_z < n_v$ \cite{Skibik2021_feasibilityGov}. Assumption~\ref{ass:GZ-Nonsingular} is made instead to simplify the approach.
\end{rem}

The reference tracking MPC strategy of \cite{Limon2008_trackingMPC} is employed to solve the specified control problem. The following parametric OCP is used to generated the MPC feedback law:
\begin{subequations} \label{eq:LMPC_OCP}
\begin{alignat}{2}
\underset{(\xi,\mu)}{\mathrm{min}}& &&||\xi_N - \bar x_v||_P^2 + \sum_{i=0}^{N-1} ||\xi_i - \bar x_v||_Q^2 + ||\mu_i - \bar u_v||_R^2\\
\mathrm{s.t.}& ~ &&~\xi_0 = x, \label{eq:ocp_cstr1} \\
& &&~\xi_{i+1} = A\xi_i + B \mu_i, ~~~ i \in \intrng{0}{N-1},\\
& &&~C \xi_i + D \mu_i \in \mc{Y},\label{eq:ocp_cstr2} ~~~~~~ i \in \intrng{0}{N-1},\\
& &&\qquad(\xi_N,v) \in \mc{T}, \label{eq:ocp_cstr3}
\end{alignat}
\end{subequations}
where $N \in \naturals$ is the prediction horizon, $\xi = (\xi_0,\dots,\xi_{N})$ and $\mu = (\mu_0,\dots,\mu_{N-1})$ are the predicted state and control sequences, $P \in \reals^{n\times n}, Q \in \reals^{n\times n}, R \in \reals^{n_u\times n_u}$ are weighting matrices, and $\mc{T} \subseteq \reals^{n+n_v}$ is a terminal set. The current state $x$ and reference command $v$ are parameters in this OCP. We use the notation $\mc{P}_N(x,v)$ to refer to problem \eqref{eq:LMPC_OCP} specified with fixed parameters $(x,v)\in\reals^{n+n_v}$. The following assumption ensures that \eqref{eq:LMPC_OCP}  can be used to generate a stabilizing feedback law.

\begin{ass} \label{ass:QR}
   The cost matrices satisfy $Q \succ 0$ and $R \succ0$.
\end{ass}
\begin{rem}
     We assume $Q \succ 0$, as opposed to $Q \succeq 0$, in order to invoke a stability result  \cite[Theorem 13.1]{borelli_BBMBook} for approximate MPC.
\end{rem}

Given $Q$ and $R$ specified according to Assumption~\ref{ass:QR}, let $P$ be positive-definite the solution to 
\begin{equation}
    P = Q+A^TPA-A^TPBK,
\end{equation}
where $K$ is the Linear Quadratic Regulator (LQR) gain 
\begin{equation} \label{eq:LQRGain}
    K = (R+B^TPB)^{-1}(B^T PA).
\end{equation}
Further, let $\mc{T}=\mc{O}_{\infty}$, where $\mc{O}_{\infty} \subseteq \reals^{n+n_v}$ is the maximum constraint admissible set for the closed-loop system under  LQR feedback \cite{gilbert1991_maxAdmissSets}. That is, $\mc{O}_{\infty}$ is the maximal set of pairs $(x,v)$ such that if LQR is applied to \eqref{eq:LTI_System} with $x_0 = x$ and a constant reference $v$, then  the constraint \eqref{eq:constraint} is satisfied.

Since $\mc{T}=\mc{O}_{\infty}$ is polyhedral under Assumption~\ref{ass:AB-Stab} \cite{gilbert1991_maxAdmissSets}, then $\mc{P}_N(x,v)$ can be written in a condensed form:
\begin{subequations} \label{eq:Condensed_LMPC_OCP}
\begin{alignat}{2} 
\underset{\mu}{\mathrm{min}}&\quad \frac12&& \mu^T H \mu + \mu^T W\theta\\
\mathrm{s.t.}& && M \mu + L \theta + b \geq 0 ,
\end{alignat}
\end{subequations}
where $\theta = (x,v)$ and $H \succ 0, W, M, L, b$ are defined in \cite{Skibik2021_feasibilityGov} as a function of the problem data in \eqref{eq:LMPC_OCP}. The set of feasible parameters for \eqref{eq:Condensed_LMPC_OCP} is 
\begin{equation}
    \Gamma_N =  \{ \theta \in \reals^{n+n_v} ~|~ \exists \mu :  M \mu + L \theta + b \geq 0  \},
\end{equation}
which is equivalent to the $N$-step backwards reachable set to $\mc{O}_\infty$. Note that both $\mc{O}_\infty$ and $\Gamma_N$ are polyhedral sets with representations that can be computed offline  \cite{gilbert1991_maxAdmissSets,Skibik2021_feasibilityGov}. The following set-valued maps are defined for convenience
\begin{align}
    \mathcal{O}_\infty(v) &=  \{ x \in \reals^{n} ~|~ (x,v) \in \mathcal{O}_\infty \}, \\
    \Gamma_N(v) &=  \{ x \in \reals^{n} ~|~ (x,v) \in \Gamma_N \}.
\end{align}

The MPC feedback policy is defined by
\begin{equation} \label{eq:MPC_Feedback_Law}
     u^* (x,v) = \Xi \mu^*(x,v),
\end{equation}
where $\mu^*(x,v)$ is the optimal solution to $\mc{P}_N(x,v)$ and $\Xi = [I_{n_u} \ 0 \ ... \ 0]$ is a matrix that selects the first control input. The following theorem details the stability and convergence properties of the closed-loop MPC system for a constant reference $v$.

\begin{thm} \label{thm:stabilityOfMPC}
(\hspace{-1sp}\cite[Theorem 4.4.2]{goodwin_ConstrainedControl}, \cite[Theorem 1]{Skibik2021_feasibilityGov}) Let Assumptions~\ref{ass:AB-Stab}-\ref{ass:QR} hold and consider the closed-loop system
\begin{equation} \label{eq:closedLoopMPCDynamics}
    x_{k+1} = A x_k + B u^*(x_k,v),
\end{equation}
starting from an initial condition $x_0$ with a constant reference command $v$. Then for all $(x_0,v) \in \Gamma_N$, all solutions satisfy: $(x_k,v) \in \Gamma_N$, $\forall k \in \posints$;  $y_k \in \mc{Y}$, $\forall k \in \posints$; and $\lim_{k \to \infty} x_k = \bar{x}_{v}$. If, in addition, $v \in \mathrm{Int} \ \mc{V}$ then $\bar{x}_{v}$ is asymptotically stable, where $\mathcal{V} = \{ v  ~|~ G_y v \in \mathcal{Y} \}$ is the set of constraint admissible references and $G_y = C G_x + D G_u$.
\end{thm}

\begin{ass} \label{ass:referenceInV}
    The target reference satisfies $r \in \mathrm{Int} \ \mc{V}$.
\end{ass}

Clearly, the feedback policy \eqref{eq:MPC_Feedback_Law} is sufficient to solve the given control problem. However, computing the solution of \eqref{eq:Condensed_LMPC_OCP} at each timestep may be  difficult in applications with fast sampling rates and/or limited computing power. To this end, a computationally governed optimization approach for solving \eqref{eq:Condensed_LMPC_OCP} is proposed in this paper. 

\section{A Log-domain Interior-point Method for Quadratic Programming}
\label{sec:LDIPM}
In this section, a brief overview of the log-domain interior-point method (LDIPM) from \cite{Permenter2021_logspaceIP} is presented.  To avoid dealing with the extra parameters in \eqref{eq:Condensed_LMPC_OCP}, we consider a generic convex quadratic program (QP) of the form:
\begin{subequations} \label{eq:genericQP}
\begin{alignat}{2} 
\underset{z}{\mathrm{min.}}&\quad \frac12&& z^T H z + c^T z\\
\mathrm{s.t.}& && Az+b \geq 0,
\end{alignat}
\end{subequations}
where $z \in \reals^p$ is the vector of decision variables, and $A \in \reals^{m \times p}$, $b \in \reals^m$, $H \in \reals^{p \times p}$, and $c \in \reals^p$ are the problem data. It is assumed that $H \succeq 0$, there exists $z \in \reals^p$ satisfying $Az +b > 0$ for all $\beta \in \reals$, the sublevel set $\{ z ~|~ \frac{1}{2} z^T H z + c^T z \leq \beta, Az + b \geq 0 \}$ is bounded, and $A^T A + H \succ 0$.

Consider the following \textit{central-path} equations for \eqref{eq:genericQP}
\begin{subequations} \label{eq:centralPath}
\begin{align}
 &A^T \lambda = Hz + c, \ \  s = Az + b, \\
 &\lambda \geq 0, \ s \geq 0, \ \   s_i \lambda_i = \eta, \ \forall i \in \naturals_{[1,m]}, \label{eq:nonNegAndCompSlack}
\end{align}
\end{subequations}
where $\eta > 0$ is a fixed homotopy parameter, $s \in \reals^m$ is the constraint slack, and $\lambda \in \reals^m$ is the vector of dual variables. Note that when $\eta =0$, these equations reduce to the Karush-Kuhn-Tucker (KKT) optimality conditions for \eqref{eq:genericQP}.  Next, consider the following logarithmic change-of-variables. Let $\gamma \in \reals^m$ and define $\lambda = \sqrt{\eta} e^\gamma$ and $s = \sqrt{\eta} e^{-\gamma}$, such that the \textit{log-domain central-path} equations are 
\begin{equation} \label{eq:logspaceCentralPath}
    \sqrt{\eta} A^T e^{\gamma} = Hz + c, \ \ \ \sqrt{\eta} e^{-\gamma} = Az + b.
\end{equation}
Note that the conditions in \eqref{eq:nonNegAndCompSlack} are  automatically satisfied by the change-of-variables. We define the  \emph{central-path} as the map $\eta \mapsto (z,\gamma)$ such that $(z,\gamma,\eta)$ satisfy \eqref{eq:logspaceCentralPath}. 

The LDIPM solves \eqref{eq:genericQP} by applying Newton's method to \eqref{eq:logspaceCentralPath} with a decreasing sequence of $\eta$. The Newton direction  $d = d(\gamma,\eta) \in \reals^m$ of the LDIPM satisfies 
\begin{subequations} \label{eq:newtonDirectDef}
\begin{align}
  \sqrt{\eta} A^T ( e^{\gamma} + e^{\gamma} \circ d) &= Hz + c, \label{eq:newtonDirectDef_1} \\
  \sqrt{\eta}( e^{-\gamma} - e^{-\gamma} \circ d) &= Az + b. \label{eq:newtonDirectDef_2}
\end{align}
\end{subequations}
The following theorem establishes uniqueness of $d(\gamma,\eta)$ and $z(\gamma,\eta)$, and provides a method of computing both.

\begin{thm} \label{thm:NewtonStep}
    (\hspace{-1sp}\cite[Theorem 2.1]{Permenter2021_logspaceIP}) For all $\gamma \in \reals^m$ and $\eta > 0$, the Newton direction $d = d(\gamma,\eta)$ and the decision variable $z=z(\gamma,\eta)$ satisfy 
    \begin{align}
        d &= \ones - \frac{1}{\sqrt{\eta}} e^{\gamma} \circ (Az + b), \label{eq:NewtonDirectionEquation} \\
        (A^T \Phi(\gamma) A +H) z &= 2 \sqrt{\eta} A^T e^\gamma - (c + A^T \Phi(\gamma) b), \label{eq:NewtonPrimalVar}
    \end{align}
    where $\ones \in \reals^m$ is a vector of ones and $\Phi(\gamma) = \mathrm{diag}(e^{2\gamma})$. Moreover, $A^T \Phi(\gamma) A +H \succ 0$.
\end{thm}

Given $\eta > 0$, iterates generated by the update rule
\begin{align}
    \gamma_{i+1} &= \gamma_i + \frac{1}{\alpha_i} d(\gamma_i, \eta), \\
    \alpha_i &= \max \left \{ 1,   \| d(\gamma_i, \eta) \|_\infty^2  \right \},
\end{align}
are globally convergent to the central-path point $(z,\gamma,\eta)$  \cite{Permenter2021_logspaceIP}. The following lemma describes conditions under which  iterates are primal-dual feasible with bounded suboptimality. 

\begin{lem} \label{lem:primalDualFeas}
  (\hspace{-1sp}\cite[Lemma 3.3]{Permenter2021_logspaceIP}) For $\eta > 0$, let $d = d(\gamma,\eta)$ and $z = z(\gamma,\eta)$. Let $\lambda = \sqrt{\eta}(e^\gamma + e^\gamma \circ d)$ and $s = \sqrt{\eta}(e^{-\gamma} - e^{-\gamma} \circ d)$. If $\| d \|_\infty \leq 1$, then $(z,s,\lambda)$ satisfy the primal-dual feasibility conditions
  \begin{equation*}
      Az + b = s, \  A^T \lambda = Hz + c, \  \lambda \geq 0, \ s \geq 0.
  \end{equation*}
  Further, $\| s \circ \lambda \|_1 = \eta (m - \| d \|^2)$.
\end{lem} 

\begin{algorithm}
\caption{Longstep($H,c,A,b,\gamma_0,\eta_0$,$\eta_f$)} \label{algo:longstep}
\begin{algorithmic}[1] 
 \renewcommand{\algorithmicrequire}{\textbf{Input:}}
 \renewcommand{\algorithmicensure}{\textbf{Output:}}
 \State $\gamma \leftarrow \gamma_0$, $\eta \leftarrow \eta_0$
 \While{$\eta > \eta_f$ \textbf{or} $\| d (\gamma,\eta) \|_\infty > 1$}
 \State $\eta \leftarrow \min \{ \eta, \inf \{ \eta > 0 : \|d(\gamma,\eta) \|_\infty \leq 1 \}     \}$
 \State $\alpha \leftarrow \max \{ 1,   \| d(\gamma, \eta) \|_\infty^2  \} $
 \State $ \gamma \leftarrow \gamma + \frac{1}{\alpha} d(\gamma, \eta)$
 \EndWhile \\
  $z \leftarrow (A^T \Phi(\gamma) A +H)^{-1} [ 2 \sqrt{\eta} A^T e^\gamma - (c + A^T \Phi(\gamma) b)],$ \\
 \Return $(z,\gamma,\eta)$
 \end{algorithmic}
 \end{algorithm}

Lemma~\ref{lem:primalDualFeas} motivates the longstep procedure described in Algorithm~\ref{algo:longstep}. At each iteration, the longstep procedure seeks to reduce $\eta$ by computing
\begin{equation} \label{eq:etaStar}
    \eta^* = \inf \{ \eta > 0 ~|~ \|d(\gamma,\eta) \|_\infty \leq 1 \},
\end{equation}
where we define $\eta^* = \infty$ when the set in \eqref{eq:etaStar} is empty. When $\eta^* < \infty$ is found, Algorithm~\ref{algo:longstep} provides an update that ensures $(z,\gamma,\eta)$ is primal-dual feasible and within a neighborhood of the central-path. Note that $\eta^*$ can be computed in ${O}(m)$ time by iterating through $2m$ linear inequalities \cite[Section 3.2.1]{Permenter2021_logspaceIP}. Moreover, the following theorem shows that Algorithm~\ref{algo:longstep} terminates globally.

\begin{thm} \label{thm:convergenceOfLDIPM}
     (\hspace{-1sp}\cite[Theorem 3.2]{Permenter2021_logspaceIP})  For any input $(\gamma_0,\eta_0,\eta_f) \in \reals^m \times \reals_{>0} \times \reals_{>0}$,  Algorithm~\ref{algo:longstep} terminates and  returns $(z,\gamma,\eta)$ with $\eta \leq \eta_f$, $\| d(\gamma,\eta) \|_\infty  \leq 1$, and
     \begin{equation*}
         Az + b \geq 0, \hspace{0.5cm} \frac{1}{2} z^T W z + c^T z \leq V^* + m \eta,
     \end{equation*}
     where $V^*$ denotes the optimal value of the QP \eqref{eq:genericQP}.
\end{thm}

\section{Application of the Log-Domain Interior Point Method to MPC}
\label{sec:warmStart}

We now consider the application of LDIPM to the MPC problem in Section~\ref{sec:problemSetting}. The following theorem and corollary demonstrate that the equilibrium $\bar{x}_v$ is asymptotically stable when MPC solutions are computed using Algorithm~\ref{algo:longstep} with a sufficiently small truncation tolerance $\eta_f$.
\begin{thm} \label{thm:stabUnderSubopt}
Let Assumptions~\ref{ass:AB-Stab}-\ref{ass:QR} hold. Let $\tilde{\mu}: (x,v) \mapsto \mu$  be a function that generates a feasible solution to $\mc{P}_N(x,v)$ satisfying
\begin{equation} \label{eq:costBoundForStab}
    J(x,v,\tilde{\mu}(x,v)) < V(x,v)  +  \| x - \bar{x}_v \|_Q^2,  
\end{equation}
for all $(x,v) \in \Gamma_N \backslash (\bar{x}_v,v)$, where $J$ and $V$ are the cost function and optimal value function of $\mc{P}_N(x,v)$. Then, consider the closed-loop dynamics
\begin{equation}  \label{eq:closedLoopSubopt}
    x_{k+1} = A x_k + B \Xi \tilde{\mu}(x_k,v),
\end{equation}
starting from an initial condition $x_0$ with a constant reference command $v$. Then for all $(x_0,v) \in \Gamma_N$, all solutions satisfy:  $(x_k,v) \in \Gamma_N$, $\forall k \in \posints$; $y_k \in \mc{Y}$, $\forall k \in \posints$; and $\lim_{k \to \infty} x_k = \bar{x}_{v}$. If, in addition, $v \in \mathrm{Int} \ \mc{V}$ then $\bar{x}_{v}$ is asymptotically stable.
\end{thm}
\begin{proof}
    The result is a direct consequence of Theorem~\ref{thm:stabilityOfMPC} and \cite[Theorem 13.1]{borelli_BBMBook}. 
\end{proof}

\begin{cor} \label{cor:eta_f_select}
   Suppose Algorithm~\ref{algo:longstep} is applied to $\mc{P}_N(x_k,v)$ and let  $\eta_{f,k}$ represent the truncation tolerance $\eta_f$ specified in Algorithm~\ref{algo:longstep} at timestep $k$. Further, let $\eta_{f,k}$ satisfy
   \begin{equation} \label{eq:eta_f_select}
      m \eta_{f,k} <  \| x_k - \bar{x}_v \|_Q^2, \quad \forall k \geq 0, 
   \end{equation}
   where $m \in \naturals$ is the number of constraints in \eqref{eq:Condensed_LMPC_OCP} (i.e. $b \in \reals^m$).  Then, any output of Algorithm~\ref{algo:longstep} satisfies Theorem~\ref{thm:stabUnderSubopt}.
\end{cor}

\begin{proof}
    Any solution ${\mu}$ generated by Algorithm~\ref{algo:longstep} is feasible and satisfies $J(x,v,{\mu}) - V(x,v) \leq m \eta_{f,k}$ by Theorem~\ref{thm:convergenceOfLDIPM}.   Hence,  \eqref{eq:costBoundForStab} is satisfied by the bound in \eqref{eq:eta_f_select}. 
\end{proof}

\begin{ass} \label{ass:etaf_tolerance}
    At all timesteps $k$,  the LDIPM truncation tolerance $\eta_{f,k}$ is chosen to satisfy \eqref{eq:eta_f_select}.  
\end{ass}


Now we consider how the LDIPM can be warm-started.  Let   $\xi_{k-1} = (\xi_{0,k-1},\dots,\xi_{N,k-1})$ and $\mu_{k-1} = (\mu_{0,k-1},\dots,\mu_{N-1,k-1})$ represent the  state and control sequences outputted by the LDIPM applied to $\mc{P}_N(x_{k-1},v)$. Denote by $\eta_{k-1} \leq \eta_{f,k-1}$ the tolerance that Algorithm~\ref{algo:longstep} truncated with at timestep $k-1$. The warm-started primal decision variable at timestep $k$ is generated by
\begin{equation} \label{eq:primalWS}
    \bar{\mu}_k = (\mu_{1,k-1} \dots, \mu_{N-1,k-1}, \bar{u}_v -K (\xi_{N,k-1} -\bar{x}_v)),
\end{equation}
where $K$ is the LQR gain in \eqref{eq:LQRGain}. Note that $\bar{\mu}_k$ is always a 
feasible solution candidate to $\mc{P}_N(x_k,v)$ as consequence of the MPC formulation \cite[Chapter 2]{Rawlings2009_MPCBook} and feasibility of $\mu_{k-1}$ at $k-1$. Thus, the corresponding slack variable satisfies $\bar{s}_k = M \bar{\mu}_k + L \theta_k + b  \geq 0$, where $\theta_k = (x_k,v)$. The log-domain variable used to initialize Algorithm~\ref{algo:longstep} is then generated by
\begin{align}
    \bar{\gamma}_k &= - \log\left( \max \left \{ \frac{\bar{s}_k}{\sqrt{\eta_{k-1}}}, 
      \epsilon_s \textbf{1} \right \} \right), \label{eq:warmStartGamma}
\end{align}
where $\epsilon_s > 0$ is a small tolerance, $\log$ operates elementwise, and $\max\{x,y\} = (\max(x_1,y_1),...,\max(x_n,y_n)) \in \reals^n$ for $x,y\in\reals^n$. The first argument in the $\max$ operator is a rearrangement of the  parameterization in \eqref{eq:logspaceCentralPath}, whereas the second argument is included to ensure that \eqref{eq:warmStartGamma} is defined when $\bar{s}_k$ has elements that are equal to zero.

Once initialized with  $\bar{\gamma}_k$,  Algorithm~\ref{algo:longstep} begins by  computing $\eta^*$ in \eqref{eq:etaStar}. That is, it finds the smallest $\eta$ that ensures the warm-start is within a neighborhood of the corresponding central-path point. In this sense, the warm-started solution only needs to be sufficiently close to \textit{some} location on the central-path. In contrast, other IPMs that are not designed with warm-starting in mind may be forced to start with a fixed (often large) initial value of $\eta$.

Note that the warm-start in \eqref{eq:warmStartGamma} may be a poor initial guess if a large reference change occurs. In such cases, the LDIPM may require several iterations to converge.  The following section introduces an approach for altering the reference command to avoid this problem.

\section{A Computational Governor for the Log-domain Interior-point Method}
\label{sec:compGov}

Next, consider an MPC feedback law defined by solving $\mc{P}_N(x_k,v_k)$ with a changing reference command defined by
\begin{equation} \label{eq:referenceParameterization}
    v_k = v_{k-1} + \kappa_k (r - v_{k-1}),
\end{equation}
where $r$ is the desired reference and $\kappa_k \in [0,1]$ is a time-varying parameter that dictates the rate at which $v_k$ converges to $r$. The parameterization in \eqref{eq:referenceParameterization} is commonly used in Scalar Reference Governors (SRGs), where $\kappa_k$ is maximized at each timestep subject to the constraint that $(x_k,v_k)$ stays inside of an invariant constraint admissible set \cite{Garone2017_RGSurvey}.

This section introduces a \emph{computational governor} (CG) that chooses $\kappa_k$ at each timestep subject to restrictions on the suboptimality and feasibility of the warm-start. In other words, the CG enforces a {computational constraint} in a similar approach to how SRGs enforce system constraints. To facilitate this development, note that the LDIPM Newton step can be parameterized by $\eta$ and $\kappa$ in the following manner.

\begin{prp} \label{prp:dParam}
Let the reference command $v=v_k$ in \eqref{eq:Condensed_LMPC_OCP} be defined by \eqref{eq:referenceParameterization} and let $d(\gamma,\eta,\kappa)$ represent the LDIPM Newton step for \eqref{eq:Condensed_LMPC_OCP} at a given $\gamma \in \reals^m$, $\eta > 0$, and $\kappa \in [0,1]$. Then, there exist $d_0(\gamma), d_1(\gamma), d_2(\gamma) \in \reals^m$ such that 
\begin{equation}
    d(\gamma,\eta,\kappa) = d_0(\gamma) + d_1(\gamma) \frac{1}{\sqrt{\eta}} + d_2(\gamma) \frac{1}{\sqrt{\eta}} \kappa. \label{eq:dBilinearForm}
\end{equation}
\end{prp}
\begin{proof}
   Note that when \eqref{eq:NewtonDirectionEquation} and \eqref{eq:NewtonPrimalVar} in Theorem~\ref{thm:NewtonStep} are written for the condensed OCP in \eqref{eq:Condensed_LMPC_OCP}, one obtains
        \begin{equation} \label{eq:dSpecified}
        d = \ones - \frac{1}{\sqrt{\eta}} e^{\gamma} \circ (M\mu  + L_x x + L_v v + b),  
    \end{equation} 
    \vspace{-\baselineskip}
    \begin{multline} \label{eq:xSpecified}
                (M^T \Phi(\gamma) M +H) \frac{1}{\sqrt{\eta}} \mu = 2 M^T e^\gamma - \frac{1}{\sqrt{\eta}} ( W_x x +  W_v v) \\ - \frac{1}{\sqrt{\eta}}  M^T \Phi(\gamma) (L_x x + L_v v + b), 
    \end{multline}
   where $L \theta = L_x x + L_v v$ and $W \theta = W_x x + W_v v$. Then by considering the parameterization of $v$ in \eqref{eq:referenceParameterization} and observing that  $\frac{1}{\sqrt{\eta}} \mu$ is affine with respect to $(\frac{1}{\sqrt{\eta}}, \frac{1}{\sqrt{\eta}} \kappa)$ and $d$ is affine with respect to $(\frac{1}{\sqrt{\eta}} \mu,\frac{1}{\sqrt{\eta}}, \frac{1}{\sqrt{\eta}} \kappa)$, it follows that $d$ is affine with respect to $(\frac{1}{\sqrt{\eta}}, \frac{1}{\sqrt{\eta}} \kappa)$.
 \end{proof}

\begin{rem}
     An efficient procedure for computing $d_0,d_1,d_2$ is provided in the Appendix.
\end{rem}

As a consequence of Proposition~\ref{prp:dParam}, the solution of the following optimization problem can be used to specify $(\eta,\kappa)$ prior to the start of Algorithm~\ref{algo:longstep} given a warm-start $\bar{\gamma}_k$
\begin{subequations} \label{eq:algorithmSearchOpt}
\begin{alignat}{2}
\underset{\eta,\kappa}{\mathrm{max}}& ~ &&~ \kappa -c \sqrt{\eta} \\
\mathrm{s.t.} & ~ &&~ \|d(\bar{\gamma}_k,\eta,\kappa) \|_\infty \leq 1, \label{eq:dInequality} \\
& ~ &&~ \eta  \in [\eta_{\mathrm{min}}, \eta_{\mathrm{max}}], \\
& ~ &&~\kappa \in [0,1], \label{eq:algorithmSearchOpt_bound1}
\end{alignat}
\end{subequations}
where $c \geq 0$ is a weighting parameter. When $c = 0$, solving \eqref{eq:algorithmSearchOpt} finds the largest reference step $\kappa^*$ that satisfies the primal-dual feasibility
conditions of  Lemma~\ref{lem:primalDualFeas}. Setting $c>0$ reduces
the reference step but improves optimality of the warm-start. 

Note that \eqref{eq:algorithmSearchOpt} can be expressed as a two-dimensional linear program (LP) in the variables $\sqrt{\eta}$ and $\kappa$ since \eqref{eq:dInequality} is equivalent to
\begin{equation*}
    (d_0 - \ones) \sqrt{\mu} + d_2\kappa \leq - d_1, \hspace{0.5cm}    -(d_0 + \ones) \sqrt{\mu} - d_2 \kappa \leq  d_1. 
\end{equation*}
To solve the resulting LP, we use Seidel's algorithm described in \cite{seidel1991_LPSolver}. We choose this algorithm because of its capability of solving low-dimensional LPs with $m$ constraints in   $O(m)$ time.



In summary, we propose the following procedure for efficiently computing MPC control actions at each timestep
\begin{enumerate}
        \item \textit{Warm-start}: Initialize $\bar{\gamma}_k$ according to \eqref{eq:warmStartGamma},
        \item \textit{Computational governor}: Set $(\eta_k,\kappa_k)$ to the solution of  \eqref{eq:algorithmSearchOpt} when feasible, otherwise set $(\eta_k,\kappa_k) = (\bar{\eta},0)$ where $\bar{\eta} \gg 1$ is a large constant value,
        \item \textit{LDIPM}:  Solve $\mc{P}_N(x_k, v_k)$ using Algorithm~\ref{algo:longstep} starting with a penalty parameter $\eta_k$ and log-space variable $\bar{\gamma}_k$.
\end{enumerate}

\section{Numerical Examples}
\label{sec:examples}

The  linear bicycle model in \cite[Section II-A]{Beal2013_CarModelForMPC} is used to demonstrate the efficacy of the computationally governed LDIPM. The system states are $x = (\beta,r,y)$ where $\beta$ is the ratio of lateral to longitudinal velocity, $r$ is the yaw rate, and $y$ is the lateral position. The control input is the steering angle $u = \delta$. The continuous-time state-space matrices $(A_c,B_c)$ are
\begin{align*}
    \left( \begin{bmatrix}
    \frac{-(C_{a f} + C_{a r})}{m U_x} &     \frac{-(a C_{a f} - b C_{a r})}{m U_x^2} - 1 & 0 \\
    \frac{-(a C_{a f} - b C_{a r})}{I_{zz}} &     \frac{-(a^2 C_{a f} + b^2 C_{a r})}{I_{zz} U_x} & 0 \\
    U_x & 0 & 0
    \end{bmatrix} ,
    \begin{bmatrix}
    \frac{ C_{a f} }{m U_x} \\ 
        \frac{a C_{a f} }{I_{zz}} \\
        0
    \end{bmatrix} \right)
\end{align*}
where  $U_x = 10$ is the constant longitudinal velocity, $m$ is the vehicle mass, $I_{zz}$ is the yaw moment of inertia, $C_{a f}$ and $ C_{a r}$ are the front and rear corning stiffness parameters, and $a$ and $b$ are the front and rear axle-CG distances. The values used for these parameters are those  in  \cite[Table I]{Beal2013_CarModelForMPC}. The system is controlled using the MPC feedback law generated by  solving $\mc{P}_N(x_k,v_k)$ using the three-step strategy described in Section~\ref{sec:compGov}. The MPC law is defined using a  sampling period of $T = 0.1$, a horizon of $N = 10$, and weight matrices $Q = \mathrm{diag}(1,1,10)$ and $R = 1$. The CG in \eqref{eq:algorithmSearchOpt} is executed with parameters $c = 1$,  $\eta_{\mathrm{min}} = 10^{-10}$, and  $\eta_{\mathrm{max}} = 10^{-2}$. The  system constraints are $\beta \in [-0.2,0.2]$, $r \in [-4,4]$, $y \in [-4,4]$, and $\delta \in [-1,1]$. 

Figure~\ref{fig:laneChangeComparison} compares the system response with and without the  CG. When no CG is used, several iterations are needed to converge to a solution when the setpoint changes at $t=0$ and $t=10$. Meanwhile, MPC solutions are obtained using only a single iteration per timestep when the CG is added to the closed-loop system. This is due to the CG maintaining a low value of $\eta^*$, thus Algorithm~\ref{algo:longstep} is initialized near the optimal solution at each timestep. Moreover, the CG introduces very little degradation in the settling time of the controller and the solution operates near the constraint boundary of $\delta$. 

To quantify the reduction in worst-case computation time when the CG is used, the experiments in Figure~\ref{fig:laneChangeComparison} were repeated 1000 times using a 2018 MacBook Pro running MATLAB R2019b. The worst-case computation times for Algorithm~\ref{algo:longstep} in Case (a) and the combination of Algorithms~\ref{algo:longstep} and solving \eqref{eq:algorithmSearchOpt} in Case (b) were recorded for each trial. Over the 1000 trials, the average of these times were (a) $10.8\pm0.9~\textrm{ms}$ and (b) $1.0\pm0.1~\textrm{ms}$. Thus, the worst-case computation time is reduced by approximately 90\% when the CG is used.

\begin{figure}
    \centering
    \subfigure[Step reference change]{\includegraphics[width=\columnwidth]{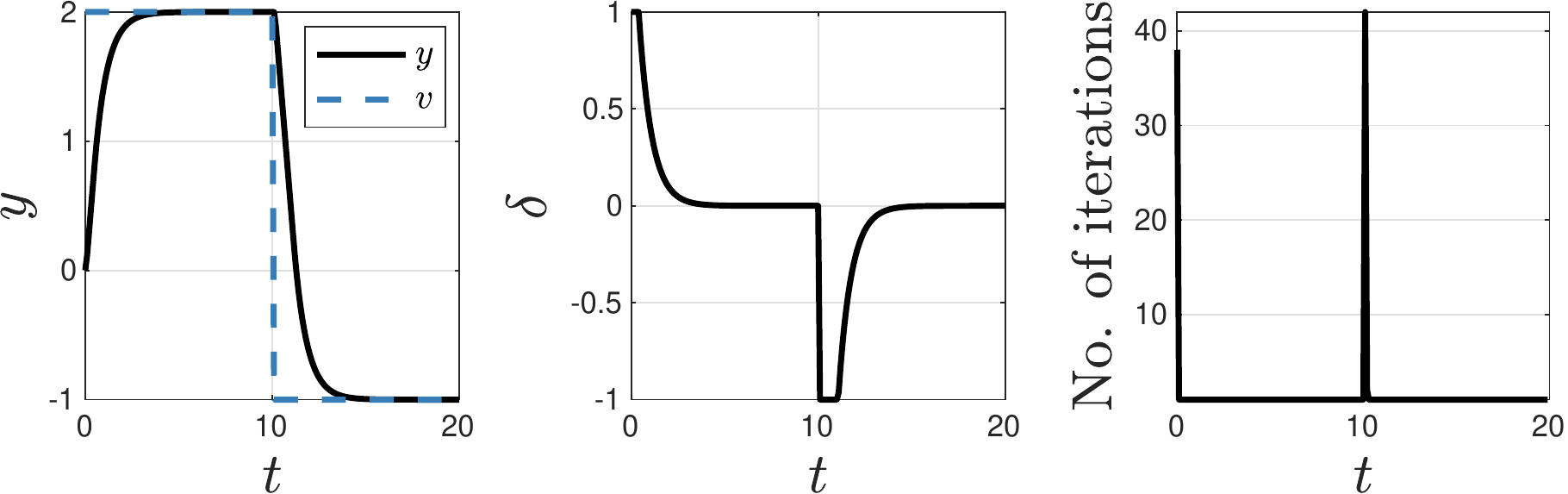}}
    
    \subfigure[Computationally governed reference change]{\includegraphics[width=\columnwidth]{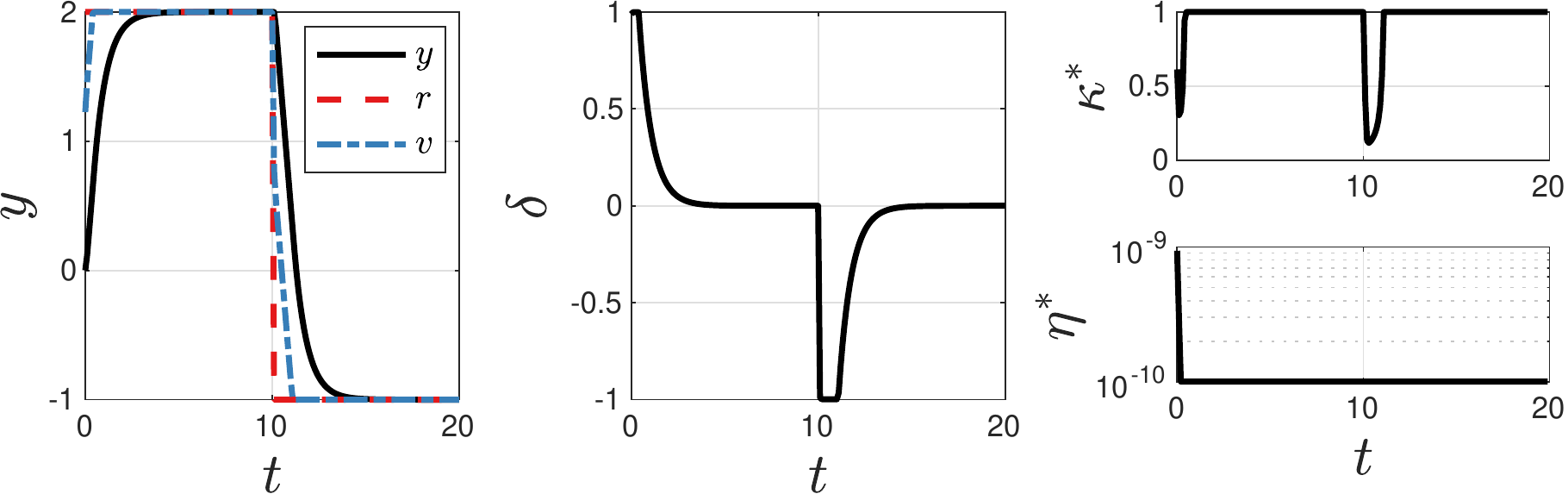}}
    \caption{Selecting $(\eta^*,\kappa^*)$ using the computational governor allows for MPC solutions to be computed using only 1 iteration per timestep. }
    \label{fig:laneChangeComparison}
\end{figure}

\section{Conclusions}

This paper introduced a computationally efficient method of implementing referencing tracking MPC. The  main novelty introduced in this paper is a computational governing scheme that alters the MPC reference command so that the  suboptimality of the optimizer initialization is bounded. Examples demonstrated that using this strategy can reduce the worst-case computation time of MPC by 90\%.   Future work will be devoted to deriving theoretical guarantees for this strategy such as finite-time convergence of the reference command and asymptotic stability of the desired equilibrium.

\section*{Appendix: Computation of $(d_0, d_1,d_2)$}

Let $\gamma \in \reals^m$  be a fixed log-domain variable  and define constants $\eta_1,\eta_2 > 0$, $\kappa_1,\kappa_2 \in [0,1]$, $\eta_1 \neq \eta_2$, $\kappa_1 \neq \kappa_2$. Consider the Newton direction equation \eqref{eq:newtonDirectDef_1}  for \eqref{eq:Condensed_LMPC_OCP} given each permutation of these parameters and a reference defined by $v' = v+\kappa(r-v)$, i.e.,
\begin{align*}
    &\sqrt{\eta_1} A^T e^\gamma \circ (\ones + \hat{d}_1) = H\mu_1 + W_x x  + W_v \left[ v + \kappa_1 (r - v) \right]  \\
 &\sqrt{\eta_1} A^T e^\gamma \circ (\ones +  \hat{d}_2) = H\mu_2 + W_x x + W_v \left[ v+ \kappa_2 (r - v) \right]    \\
 &\sqrt{\eta_2} A^T  e^\gamma \circ (\ones +  \hat{d}_3) = H\mu_3 + W_x x + W_v \left[ v + \kappa_1 (r - v) \right] \\
 &\sqrt{\eta_2} A^T e^\gamma \circ (\ones + \hat{d}_4) = H\mu_4 + W_x x + W_v \left[ v + \kappa_2 (r - v) \right]
\end{align*}
By combining these four equations and using some algebraic manipulation, one can  arrive at the equation,
\begin{equation}
            \sqrt{\eta} A^T e^\gamma \circ ( \ones + d ) = H  \mu
            +  W_x x  \\ + W_v [ v + \kappa   (r - v) ], \label{eq:bigNewtonEq_1}
\end{equation}
where 
\begin{alignat}{2}
    &\kappa = p \kappa_1 + (1-p)\kappa_2, && {\mu} = q \bar{\mu}_1 + (1-q) \bar{\mu}_2, \nonumber \\
    &d = q \bar{d}_1 + (1-q) \bar{d}_2,   &&   \frac{1}{\sqrt{\eta}} = q \frac{1}{\sqrt{\eta_1}} + (1-q) \frac{1}{\sqrt{\eta_2}}, \nonumber \\
    &\bar{d}_1 = p \hat{d}_1 + (1-p) \hat{d}_2, \hspace{0.75cm} &&\bar{\mu}_1 = p {\mu}_1 + (1-p) {\mu}_2,  \nonumber \\
    &\bar{d}_2 =  p \hat{d}_3 + (1-p) \hat{d}_4, \hspace{0.75cm}  &&  \bar{\mu}_2  = p {\mu}_3 + (1-p) {\mu}_4,  \label{eq:collectionOfEqs}
\end{alignat}
and where $p,q \in \reals$ are constants that we will specify later. Repeating the exact same procedure, but instead starting with the four equations for \eqref{eq:newtonDirectDef_2} gives
\begin{equation}
            \sqrt{\eta} e^{-\gamma} \circ \left( \ones - d \right) = M {\mu} + b +  L_x x \\ +  L_v[v_0 + \kappa  (r - v_0) ]. \label{eq:bigNewtonEq_2}
\end{equation}
By comparing  \eqref{eq:bigNewtonEq_1} and \eqref{eq:bigNewtonEq_2} to \eqref{eq:newtonDirectDef_1} and \eqref{eq:newtonDirectDef_2}, one can observe that  ${\mu}$ and $d$ defined in \eqref{eq:collectionOfEqs} are the primal variable and Newton step for parameters $\kappa$ and $\eta$  defined in \eqref{eq:collectionOfEqs}.

So, for a given $(\eta,\kappa)$ one can solve for the Newton step $d(\gamma,\eta,\kappa)$ by defining $p = a_0 + a_1 \kappa$ and $q = b_0 + b_1 \eta^{-1/2}$, where $a_0 = -\kappa_2(\kappa_1 - \kappa_2)^{-1}$, $a_1 = (\kappa_1 - \kappa_2)^{-1}$,  $b_0 = -\eta_2^{-1/2}(\eta^{-1/2}_1 - \eta^{-1/2}_2)^{-1}$, $b_1 = (\eta^{-1/2}_1 - \eta^{-1/2}_2)^{-1}$. Substituting  these expressions for $p$ and $q$ into the equation for $d$ in \eqref{eq:collectionOfEqs}  yields, after some algebraic manipulation is performed,
\begin{equation} \label{eq:dWithExtraTerm}
    d = d_0 + d_1 \frac{1}{\sqrt{\eta}} + d_2 \frac{\kappa}{\sqrt{\eta}} + d_3 \kappa,
\end{equation}
where
\begin{alignat}{2}
    & d_0 = b_0 c_1 + (1 - b_0)c_3,  \hspace{0.7cm} &&d_1 = b_1(c_1 - c_3), \nonumber \\
    & d_2 = b_1 (c_2 - c_4),  && d_3 = b_0 c_2 + (1 - b_0) c_4, \nonumber \\
    &c_1 = a_0 \hat{d}_1 + (1 - a_0) \hat{d}_2,  && c_2 = a_1(\hat{d}_1 - \hat{d}_2), \nonumber \\
    &c_3 = a_0 \hat{d}_3 + (1 - a_0) \hat{d}_4, && c_2 = a_1(\hat{d}_3 - \hat{d}_4). \label{eq:computingTheDs}
\end{alignat}
Further, we must have that $d_3 = 0$ according to  \eqref{eq:dBilinearForm}, so we can eliminate the need to directly compute one Newton direction (e.g. $\hat{d}_4$) by using the equality $d_3 = 0$ to obtain
\begin{equation} \label{eq:d4}
    \hat{d}_4 = b_0({1-b_0})^{-1}{(\hat{d}_1 - \hat{d}_2)} + \hat{d}_3.
\end{equation}

Thus, $d_0$, $d_1$, and $d_2$  in Proposition~\ref{prp:dParam} can be computed by defining  constants $\eta_1,\eta_2,\kappa_1,\kappa_2$ and computing the sampled Newton direction $\hat{d}_i$ for three of the four permutations by exploiting the common factorization.

\bibliography{bibl}

\begin{thebibliography}{10}

\bibitem{Rawlings2009_MPCBook}
J.~Rawlings and D.~Q. Mayne, {\em Model Predictive Control: Theory and Design}.
\newblock Madison, WI: {Nob Hill Publishing}, 2009.

\bibitem{borelli_BBMBook}
F.~Borrelli, A.~Bemporad, and M.~Morari, {\em Predictive Control for Linear and
  Hybrid Systems}.
\newblock Cambridge University Press, 2017.

\bibitem{kouvaritakis2016model}
B.~Kouvaritakis and M.~Cannon, ``Model predictive control: {Classical}, robust
  and stochastic,'' {\em Switzerland: Springer International Publishing}, 2016.

\bibitem{goodwin_ConstrainedControl}
G.~C. Goodwin, M.~M. Seron, and J.~A. De~Don{\'a}, {\em Constrained Control and
  Estimation: An Optimisation Approach}.
\newblock Springer Publishing Company, Incorporated, 1st~ed., 2010.

\bibitem{Domahidi2013_SOCPSolver}
A.~Domahidi, E.~Chu, and S.~P. Boyd, ``{ECOS: An SOCP solver for embedded
  systems},'' {\em 2013 European Control Conference (ECC)}, pp.~3071--3076,
  2013.

\bibitem{Patrinos2014_APGM}
P.~Patrinos and A.~Bemporad, ``An accelerated dual gradient-projection
  algorithm for embedded linear model predictive control,'' {\em IEEE
  Transactions on Automatic Control}, vol.~59, no.~1, pp.~18--33, 2014.

\bibitem{Liao2020_FBStab}
D.~Liao-McPherson and I.~Kolmanovsky, ``{FBstab: A proximally stabilized
  semismooth algorithm for convex quadratic programming},'' {\em Automatica},
  vol.~113, p.~108801, 2020.

\bibitem{frison2020_HPIPM}
G.~Frison and M.~Diehl, ``{HPIPM: a high-performance quadratic programming
  framework for model predictive control},'' {\em IFAC-PapersOnLine}, vol.~53,
  no.~2, pp.~6563--6569, 2020.
\newblock 21st IFAC World Congress.

\bibitem{diehl2005_nominalstabofRTI}
M.~{Diehl}, R.~{Findeisen}, F.~{Allgower}, H.~G. {Bock}, and J.~P. {Schloder},
  ``Nominal stability of real-time iteration scheme for nonlinear model
  predictive control,'' {\em IEE Proceedings - Control Theory and
  Applications}, vol.~152, pp.~296--308, May 2005.

\bibitem{liao2019time}
D.~Liao-McPherson, M.~Nicotra, and I.~Kolmanovsky, ``Time-distributed
  optimization for real-time model predictive control: Stability, robustness,
  and constraint satisfaction,'' {\em Automatica}, vol.~117, p.~108973, 2020.

\bibitem{zanelli2020lyapunov}
A.~Zanelli, Q.~Tran-Dinh, and M.~Diehl, ``{A Lyapunov function for the combined
  system-optimizer dynamics in inexact model predictive control},'' {\em
  Automatica}, vol.~134, p.~109901, 2021.

\bibitem{Liao-2020_LQMPCISS}
D.~Liao-McPherson, T.~Skibik, J.~Leung, I.~V. Kolmanovsky, and M.~M. Nicotra,
  ``An analysis of closed-loop stability for linear model predictive control
  based on time-distributed optimization,'' {\em IEEE Transactions on Automatic
  Control}, pp.~1--1, 2021.

\bibitem{leung2021_TDMPC-ROA_ACC}
J.~Leung, D.~Liao-McPherson, and I.~V. Kolmanovsky, ``A computable
  plant-optimizer region of attraction estimate for time-distributed linear
  model predictive control,'' in {\em 2021 American Control Conference (ACC)},
  pp.~3384--3391, 2021.

\bibitem{richter2011computational}
S.~Richter, C.~N. Jones, and M.~Morari, ``Computational complexity
  certification for real-time {MPC} with input constraints based on the fast
  gradient method,'' {\em IEEE Transactions on Automatic Control}, vol.~57,
  no.~6, pp.~1391--1403, 2011.

\bibitem{rubagotti2014_stabilizingLMPC}
M.~{Rubagotti}, P.~{Patrinos}, and A.~{Bemporad}, ``Stabilizing linear model
  predictive control under inexact numerical optimization,'' {\em IEEE
  Transactions on Automatic Control}, vol.~59, pp.~1660--1666, June 2014.

\bibitem{Zeilinger2014_robustMPC}
M.~N. Zeilinger, D.~M. Raimondo, A.~Domahidi, M.~Morari, and C.~N. Jones, ``On
  real-time robust model predictive control,'' {\em Automatica}, vol.~50,
  no.~3, pp.~683--694, 2014.

\bibitem{Feller2017_stabilizingIterations}
C.~Feller and C.~Ebenbauer, ``A stabilizing iteration scheme for model
  predictive control based on relaxed barrier functions,'' {\em Automatica},
  vol.~80, pp.~328--339, 2017.

\bibitem{bemporad2015_anytimeMPC}
A.~Bemporad, D.~Bernardini, and P.~Patrinos, ``A convex feasibility approach to
  anytime model predictive control,'' {\em arXiv preprint arXiv:1502.07974},
  2015.

\bibitem{Skibik2021_feasibilityGov}
T.~Skibik, D.~Liao-Mc~Pherson, T.~Cunis, I.~V. Kolmanovsky, and M.~M. Nicotra,
  ``A feasibility governor for enlarging the region of attraction of linear
  model predictive controllers,'' {\em IEEE Transactions on Automatic Control},
  2021.

\bibitem{Limon2008_trackingMPC}
D.~Limon, I.~Alvarado, T.~Alamo, and E.~Camacho, ``{MPC} for tracking piecewise
  constant references for constrained linear systems,'' {\em Automatica},
  vol.~44, no.~9, pp.~2382--2387, 2008.

\bibitem{gilbert1991_maxAdmissSets}
E.~G. {Gilbert} and K.~T. {Tan}, ``Linear systems with state and control
  constraints: the theory and application of maximal output admissible sets,''
  {\em IEEE Transactions on Automatic Control}, vol.~36, no.~9, pp.~1008--1020,
  1991.

\bibitem{Permenter2021_logspaceIP}
F.~Permenter, ``Log-domain interior-point methods for quadratic programming,''
  2021.
\newblock Online at
  {http://www.optimization-online.org/DB{\_}HTML/2021/09/8600.html}.

\bibitem{Garone2017_RGSurvey}
E.~Garone, S.~{Di Cairano}, and I.~Kolmanovsky, ``Reference and command
  governors for systems with constraints: A survey on theory and
  applications,'' {\em Automatica}, vol.~75, pp.~306--328, 2017.

\bibitem{seidel1991_LPSolver}
R.~Seidel, ``Small-dimensional linear programming and convex hulls made easy,''
  {\em Discrete \& Computational Geometry}, vol.~6, no.~3, pp.~423--434, 1991.

\bibitem{Beal2013_CarModelForMPC}
C.~E. Beal and J.~C. Gerdes, ``Model predictive control for vehicle
  stabilization at the limits of handling,'' {\em IEEE Transactions on Control
  Systems Technology}, vol.~21, no.~4, pp.~1258--1269, 2013.

\end{thebibliography}

\end{document}